\documentclass[graybox]{svmult}

\usepackage{helvet,amsfonts,eufrak,amssymb}         
\usepackage{courier}        
\usepackage{type1cm}        
\usepackage{makeidx,epsfig}         
\usepackage{graphicx,graphics}        
\usepackage{multicol}        
\usepackage[bottom]{footmisc}

\makeindex             

\def\esup{\sup_\mu} 
\def\diam{diam}
   \def\a{\alpha} 
\def\e{\varepsilon}
\def\ep{\varepsilon}
\def\sk{\smallskip}
\def\mk{\medskip}

\renewcommand{\liminf}{\mathop{{\underline {\hbox{{\rm lim}}}}}}

\def\Z{\mathbb{Z}} 
\def\R{\mathbb{R}} 
\def\T{\mathbb{T}} 
\def\N{\mathbb{N}}

\def\Sf{\EuFrak{SFT}}

\def\1{\mathbf 1}

\def\emptyset{\varnothing}

\def\uy{\underline{y}}
\def\supp{\mbox{Supp } (}
\def\bequ{\begin{equation}}
\def\nequ{\end{equation}}

\def\bdef{\begin{defn}}
\def\ndef{\end{defn}}

\def\bthm{\begin{thm}}
\def\nthm{\end{thm}}

\def\bprop{\begin{prop}}
\def\nprop{\end{prop}}

\def\brmk{\begin{remarks}}
\def\nrmk{\end{remarks}}

\def\bdes{\begin{description}}
\def\ndes{\end{description}}

\begin{document}

\title*{On measures resisting multifractal analysis}
 \author{J\"org Schmeling and St{\'e}phane Seuret}
\institute{J\"org  Schmeling \at   Mathematics Centre for Mathematical Sciences,
Lund Institute of Technology, Lund University Box 118 SE-221 00 Lund, Sweden 
\email{joerg@maths.lth.se}
\and St\'ephane Seuret \at LAMA, CNRS UMR 8050,
Universit\'e Paris-Est  Cr\'eteil, 61 Avenue du
G\'en\'eral de Gaulle, 94010 Cr\'eteil Cedex, France
\email{seuret@u-pec.fr}}

%
%
\maketitle

\abstract*{Any ergodic measure of a smooth map on a compact manifold has a multifractal spectrum with one point - the dimension of the measure itself - at the diagonal. We will construct examples where this fails in the most drastic way for invariant measures invariant under  linear maps of the circle. } 

\abstract{Any ergodic measure of a smooth map on a compact manifold has a multifractal spectrum with one point - the dimension of the measure itself - at the diagonal. We will construct examples where this fails in the most drastic way for invariant measures invariant under  linear maps of the circle. }

\medskip

\begin{center}
{\em Dedicated to Victor Afraimovich on the occasion of his 65th birthday.}
\end{center}
 
\section{Introduction}\label{sec1}

Let $\mu$ be a probability measure on a metric space $(X,d)$. For $x\in\supp\mu)$
define
\[
d_\mu(x):=\liminf_{r\to 0}\frac{\log\mu (B(x,r))}{\log r}
\]
where $B(x,r)$ is the ball of radius $r$ centered at $x$. For $\a\ge 0$ we will consider the level sets
\[
D_\mu(\a):=\{x\in\supp\mu)\, :\, d_\mu(x)=\a\}.
\]
The multifractal spectrum of $\mu$ is given by
$$
f_\mu(\a):=  \left\{\begin{tabular}{ll} \smallskip$f_\mu(\a)=-\infty  $  & if $ D_\mu(\a) =\emptyset $,\\ \smallskip
$ \dim_H D_\mu(\a)$& otherwise.\end{tabular}\right.
$$
 
The dimension of a measure $\mu$ is defined as 
\begin{equation}
\label{defdimmu}
\dim_H\mu:=\inf\{\dim_HZ\, :\, \mu(Z)=1\}.
\end{equation}
It is well-known that
\begin{equation}\label{esup}
\dim_H\mu=\esup d_\mu(x),
\end{equation}
$\esup$ standing for the $\mu$-essential supremum. Hence it is likely that the graph of the function $f_\mu$ touches the diagonal at $\alpha=\dim_H\mu$. This phenomenon happens for any Gibbs measure associated with a H\"older potential invariant under a dynamical system, and we may wonder if this is a general property for measures, invariant measures or ergodic measures. In this note we will give examples of invariant measures that have a multifractal spectrum as far as possible off the diagonal. Indeed these measures can be chosen to be invariant under linear transformations of the circle. We will also remark that the same situation does not occur for ergodic measures, for which the multifractal spectrum always touches the diagonal.

\begin{theorem}\label{main}
For given $(a,b)\in [0,1]$ there is a probability measure $\mu$ supported on a compact Cantor set $K\subset [0,1]$ with the following properties:
\begin{itemize}
\item[i)] $\mu (I)>0$ for all non-empty open sets (in the relative topology) in $K$,
\item[ii)] $\dim_H\mu =b$,
\item[iii)]  if $S= \{ d _\mu(x): x\in K\}$  is the support of the multifractal spectrum of $\mu$, then $a=\min S$ and $b=\max S$. In particular, $d_\mu (x)\in [a,b]$ for all $x\in\supp\mu) =K$,
\item[iv)] $D_\mu(\a)$ contains at most one point for all $\a\ge 0$. 
 
\end{itemize}
\end{theorem}

 The exponent at which the multifractal spectrum touches the diagonal, when it exists,  is  characterized by many properties. Let us introduce two other spectra for measures.

\begin{definition}
 For all integers $j\geq 1$, we denote by $\mathcal{G}_j$ the set of dyadic intervals of generation $j$ included in $[0,1]$, i.e. the intervals   $[k2^{-j}, (k+1)2^{-j})$, $k\in \{0,\cdots,2^j-1\}$. The Legendre spectrum of a Borel probability measure whose support is included in the interval $[0,1]$ is the map
$$ L_\mu:  \  \alpha \geq 0 \mapsto     \inf _{q\in \mathbb{R}} \ (\,  q\alpha - \tau_\mu(q) \,)   \  \ \ \in \R^+\cup\{-\infty\},$$
  where the scaling function $\tau_\mu$ is defined for $q\in \R$ as 
  $$ \tau_\mu(q):=\liminf_{j\to +\infty}    \frac{1}{-j}  \log_2   \sum_{I\in \mathcal{G}_j}  \mu(I)^q,$$
the sum being taken over the dyadic intervals with non-zero $\mu$-mass. 
\end{definition}

The Legendre  spectrum is always defined on some interval $I \subset \R^+\cup\{+\infty\}$ (the extremal exponents may or may not belong to this interval), and is concave on its support. It is a trivial matter that there is at least one exponent $\alpha_\mu \geq 0$ such that 
\begin{equation}
\label{Ltouches}
L_\mu(\alpha_\mu)=\alpha_\mu.
\end{equation}
Comparing (\ref{defdimmu}), (\ref{esup}) and (\ref{Ltouches}), obviously when there is a unique exponent such that $f_\mu(\alpha)=\alpha$, then this exponent is also the dimension of the measure $\mu$ and also the one satisfying (\ref{Ltouches}). 

 \begin{definition}
 The large deviations spectrum of a Borel probability measure whose support is included in the interval $[0,1]$ is defined as
$$ LD_\mu(\alpha) = \lim_{\e \to 0}  \   \liminf _{j\to \infty}  \  \frac{  \log_2  N_j(\alpha,\e) }  {j}$$
where
\begin{equation}
\label{defNJ}
N_j(\alpha,\e)\!:=\#\! \left\{  I \in \mathcal{G}_j : 2^{-j  (\a+\e)}  \leq   \mu(I )  \leq 2^{-j  (\a-\e)}   \right\} \!.
\end{equation}
\end{definition}

By convention, if $N_j(\alpha,\e)=0$ for some $j$ and $\e$, then $ LD_\mu(\alpha) =-\infty$.

This spectrum describes the asymptotic behavior of the number of dyadic intervals of $\mathcal{G}_j$  having a given $\mu$-mass.  The fact that the values of the large deviations spectrum are accessible for real data (by algorithms based on  log-log estimates) makes it interesting from a practical standpoint. In the paper \cite {RIEDI}   for instance,   it is proved that the concave hull of $f_\mu$ coincides with the Legendre spectrum of $\mu$  on the support of this Legendre spectrum.
One always has for all exponents $\alpha\geq 0$
$$f_\mu(\alpha) \leq LD_\mu(\alpha) \leq L_\mu(\alpha),$$ and  when the two spectra $f_\mu$ and $L_\mu$ coincide at some $\alpha\geq 0$, one says that the {\em multifractal formalism} holds at $\alpha$. 
Actually, when the multifractal formalism holds, the three spectra (multifractal, large deviations and Legendre) coincide.

For the measure we are going to construct, the multifractal formalism does not hold at $\alpha_\mu$, nor at any exponent. 
This is the reason why we claim that this measure is "as far as possible" from being multifractal.

\begin{theorem}\label{main2}
For the measure $\mu$ of Theorem \ref{main},   we have: 
\begin{itemize}
\item[i)] $f_\mu(\alpha) = 0$ for every $\alpha \in S$, and  $f_\mu(\alpha) = - \infty$ for every $\alpha \in [a,b]\setminus S$,

\item[ii)] $ \ LD_\mu(\alpha) = \alpha$ for every $\alpha \in S$, and  $LD_\mu(\alpha) = - \infty$ for every $\alpha \in \mathbb{R}_+\setminus S$,
\item[iii)] $ \ L_\mu(\alpha) = \alpha$ for every $\alpha \in [a,b]$, and is $-\infty$ elsewhere. \\
The scaling function of $\mu$ is 
\[\tau_\mu(q) = \left\{\begin{tabular}{ll}  $b(1-q$) & $\mbox { if } q\leq 1$\\ $a(1-q)$ & $\mbox{ if } q>1$.
\end{tabular}\right. \]
\end{itemize}

\end{theorem}
Hence the three spectra differ very drastically.

 \bigskip

 The article is organized as follows. Section \ref{ergodic} discusses the difference between ergodic and invariant measures as regards to our problem. Section \ref{sec_main} contains the construction of a measure $\mu$ supported by a Cantor set whose multifractal spectrum does not touch the diagonal. In Section \ref{SecLD}, we compute the Legendre and the large deviations spectra of $\mu$. 
 
 
%

%


\section{Ergodic and Invariant measures}\label{ergodic}

First we prove that the multifractal spectrum of ergodic measures always touches the diagonal.

\begin{theorem}\label{ergspec}
Let $\mu$ be an ergodic probability measure invariant under a $C^1$--diffeomorphism $T$ of a compact manifold $M$. Then $f_\mu(\dim_H\mu)=\dim_H\mu$.
\end{theorem}
\begin{proof}
Since $T$ is a smooth diffeomorphism on a compact manifold both the norm $\|D_xT\|$ and the conorm $\|(D_xT)^{-1}\|^{-1}$ are bounded on $M$. Hence, there is a $C>1$ such that for any $x\in M$ and any $r>0$
\[
B(Tx,C^{-1}r)\subset T(B(x,r))\subset B(Tx,Cr).
\]
This immediately implies that $d_\mu$ is a (of course measurable) invariant function. By ergodicity of $\mu$ it takes exactly one value for $\mu$--a.e. $x\in M$. By~(\ref{esup}) this value equals $\dim_H\mu$. 

\end{proof}

Contrarily to what happens for ergodic measures, a general invariant measure behaves as bad as a general probability measure. We will illustrate this on a simple example.
%
Consider the (rational) rotation $x\to x+\frac12 \pmod 1$ on the unit circle $\T=\R/\Z$. This transformation is not uniquely ergodic and has plenty of invariant measures. By the Ergodic Decomposition Theorem the space $M_{inv}$ of invariant measures equals
\[
\left\{\mu:=\frac12\int_{[0,1/2]}(\delta_x+\delta_{x+1/2})\, d\nu (x) \, :\, \nu \mbox{ is a probability measure on $[0,1/2)$}\right\}.
\]
W.l.o.g. assume that $x\in [0,1/2)$ and $r>0$ is sufficiently small. Then
\[
\mu (B(x,r))=\frac12\int_{B(x,r)}\, d\nu = \frac12\nu (B(x,r)).
\]
Hence,
$$
d_\mu(x)=d_\nu(x)\qquad\mbox{and}\qquad f_\mu(\a)=f_\nu(\a).
$$

In particular, using the example built in the following sections, there is a measure with a multifractal spectrum not touching the diagonal, which can not happen for an ergodic measure.

%



\section{The main construction}\label{sec_main}
 
We will represent the numbers $x$ in $[0,1]$ by their dyadic expansion, i.e. $x=\sum_{j\geq 1} x_j2^{-j}$, $x_j\in \{0,1\}$. The construction will avoid the dyadic numbers so that no ambiguity will ocur.  For $x\in [0,1]$, the prefix of order $J$ of $x$ is $x_{|J} =  \sum_{j= 1}^J x_j2^{-j}$. We will also use the notation $x=x_1x_2\cdots x_j\cdots$, and $x_{|J}= x_1\cdots x_J$.

A cylinder $C=[x_1x_2\cdots x_J]$ consists of the real numbers $x$ with prefix of order $J$ equal to $x_1x_2\cdots x_J$. The length $J$ of such a cylinder is denoted by $|C|=J$.  We denote by $\mathcal{G}_J$ the cylinders of  length $J$.  The concatenation of two cylinders $C_1=[x_1 \cdots x_J]$ and $C_2=[y_1 \cdots y_{J'}]$ is the cylinder $[x_1 \cdots x_J y_1 \cdots y_{J'}]$, and is denoted $C_1C_2$.

\sk

We stand some facts about subshifts of finite type.
First we remark that given any non-empty interval $I\subset [0,\log 2]$ there is a mixing subshift of finite type that has entropy $h_{top}(\Sigma)\in I$. We denote the set of all mixing subshifts of finite type by $\Sf$. For $\Sigma\in\Sf$ the unique measure of maximal entropy is denoted by $\mu_\Sigma$. By standard theorems,  there is a constant $M_\Sigma$ depending only on $\Sigma$ such that t for any cylinder $C_J \in\Sigma$ of length $J$
\[
M_\Sigma^{-1}  \, 2^{-h_{top}(\Sigma)J}<\mu_\Sigma(C_J)< M_\Sigma \, 2^{-h_{top}(\Sigma)J}.
\]
In addition, for the same constant $M_\Sigma$, we have
$$M_\Sigma^{-1}  \, 2^{h_{top}(\Sigma)J}<\#\{C\in \mathcal{G}_J: C\in \Sigma\} < M_\Sigma \, 2^{h_{top}(\Sigma)J}.
$$
Of course the two last double-sided inequalities are complementary.
\medskip

We now proceed to the construction of the measure $\mu$ of Theorem \ref{main}.

\medskip

{\bf Step 1:} We fix a map $\Sigma\colon \bigcup_{J=1}^\infty \{0,1\}^J\to \Sf$ with the property that
\[
h_{top}(\Sigma(y_1\cdots y_J))\in (b-a)\left[\sum_{j=1}^{J-1}\frac{2y_j}{3^j}+\left(\frac{2y_J}{3^J},\frac{2y_J+1}{3^J}
\right) \right]+a.
\]
This map is increasing in the sense that if $t_1\cdots t_J < y'_1\cdots y'_J$ (using the lexicographic order), then $h_{top}(\Sigma(y_1\cdots y_J)) < h_{top}(\Sigma(y'_1\cdots y'_J))$.

\bigskip

{\bf Step 2:} For $\Sigma\in\Sf$ and $\delta >0$, define
$$
 N(\Sigma,\delta)  \!   := \!  \min\left\{ \!J\in\N  : \!\! \left\{
  \begin{tabular}{ll}    \smallskip   $  \ \forall \ j\geq J, \ \  \forall  \ C_j \in\Sigma \mbox{ of length $j$},$ \\  \smallskip
   $\ 2^{-(h_{top}(\Sigma)+\delta)j}<\mu_\Sigma(C_j )< 2^{-(h_{top}(\Sigma)-\delta)j}  $ \\ \smallskip
 $  \mbox{ and }  \forall \ j\geq J, $\\ \smallskip
$ \ 2^{ (h_{top} (\Sigma)-\delta) j}  <  \# \{C\in \mathcal{G}_j: C\in \Sigma \}< 2^{(h_{top}(\Sigma)+\delta)j}  \ $
  \end{tabular}  \right.
 \!   \!  \right\}     \! .
$$
The numbers $N(\Sigma,\delta)$ allow us to estimate the time we have to wait until we see an almost precise value of the local entropy for a given subshift of finite type. Moreover, we have also a control of the number of cylinders of length $j\geq N(\Sigma,\delta)$ in $\Sigma$.
We then set
$$
 \delta_{J}  = \frac{b-a}{6 \cdot 2^J} \ \mbox{ and } \ \  N_J  :=   \max\Big\{N\Big(\Sigma(y_1\cdots y_J) , \delta_{J} \Big)\, :\, y_1\cdots y_J \in \{0,1\}^J\Big\}.
$$

\bigskip

{\bf Step 3:} Let   $y_1\cdots y_J\in \{0,1\}^J$. For a given cylinder $C_j$ of length $j$ in $\Sigma(y_1\cdots y_J)$,    there is a smallest integer $m_{C_j}$ for which for every cylinder $C'_m$ of length $m\geq m_{C_j}$ in $\Sigma(y_1...y_{J-1})$, we have
\begin{eqnarray}
\label{eq1}
 \ \  2^{-(h_{top}(\Sigma(y_1\cdots y_{J-1}))+\delta_J)(m+j)} < && \hspace{-6mm} \mu_{\Sigma(y_1\cdots y_{J-1})}(C'_m)\cdot\mu_{\Sigma(y_1\cdots y_J)}(C_j)\\ 
\nonumber
 && \ \ \ <   \   2^{-(h_{top}(\Sigma(y_1\cdots y_{J-1}))-\delta_J)(m+j)}.
\end{eqnarray}  
This property holds, since  we know that it holds for large $m$.

Then,  let  $$m_j:= \max\{m_{C_j} : C_j \in \Sigma(x_1\cdots x_J) \mbox{ and } |C_j|=j\}.$$
By construction, for every cylinder  $C_j$ of length $j$,  for every integer $m \geq m_j$, for every  cylinder  $C'_m \in \Sigma(y_1...y_{J-1})$, (\ref{eq1}) is true.

\smallskip

Then, we set  
$$M(y_1\cdots y_J) := \max \Big\{m_j : j \in \{1,2,\cdots, N_J\} \Big\}.$$

The numbers $M(y_1\cdots y_J)$ allow us to estimate how long the cylinders in the prefix subshift have to be to control the local entropy at a concatinated cylinder. 

 
 Finally, for every $J\geq 1$, we define the integer
\begin{eqnarray*}
M_J & := & \max\Big\{M\big(y_1\cdots y_J  \big)\, :\, y_1\cdots y_J \in \{0,1\}^J \Big\}.
\end{eqnarray*}

\bigskip

{\bf Step 4:} Choose a lacunary sequence $(L_J)_J$ with
\[
\frac{L_J}{\sum_{j=1}^J M_j+ N_j}\ge 2 \ \ \mbox{ and} \ \   \frac{L_{J+1}}{L_J}  \geq  \frac{2}{ \delta_{J+1} }  .
\]

Now we are ready to proceed with the construction of the measure $\mu$.

\bigskip

{\bf Step 5:} We will construct the measure by induction on dyadic cylinders. We set $K_1:=[0,1]$ and start with labelling the cylinder $[0]$ with $y_1=0$ and $[1]$ with $y_1=1$. For a subshift of finite type $\Sigma\in\Sf$ we denote by $\Sigma\vert_J$ all non-empty dyadic cylinders in $\Sigma$ of length $J\in\N$. Now we define
\[
K_2:=[0]\Sigma (0)\vert_{L_2}\cup [1]\Sigma (1)\vert_{L_2}.
\]
We will label a cylinder  $C_{L_2+1}$ in $[0]\Sigma (0)\vert_{L_2}$ (a similar labelling for $[1]\Sigma (1)\vert_{L_2}$)  by $y_1y_2(C_{L_2+1})=00$ iff
\[
C_{L_2+1}\cap \Big[\min\{ x\in [0]\Sigma (0)\vert_{L_2}\},\min\{ x\in [0]\Sigma (0)\vert_{L_2}\}+\frac12\diam [0]\Sigma (0)\vert_{L_2}\Big]\ne\emptyset,
\]
and by $y_1 y_2(C_{L_2+1})=01$ else. This way we have that for every $y_1y_2\in \{0,1\}^2$, 
\[
\diam\left(\bigcup_{y_1 y_2(C_{1+L_2})=y_1y_2} C_{1+L_2}\right)\le\frac14.
\]

Assume that for $J\geq 2$, we have defined $K_J$ as the union of cylinders of length $1+L_2+\cdots +L_J$ labelled by binary sequences $y_1\cdots y_J$ of length $J$. Moreover assume that for the defining cylinders of $K_J$, we managed the construction so that $y_1 \cdots y_J \in \{0,1\}^J$, 
\[
\diam\left(\bigcup_{y_1\cdots y_J(C_{1+L_2+\cdots +L_J})=y_1\cdots y_J} C_{1+L_2+\cdots +L_J}\right)\le\frac1{2^J}.
\]
We define the Cantor set at the $J+1$-th generation as
\[
K_{J+1}:=\bigcup_{ y_1\cdots y_ J \in\{0,1\}^J}\, \,\bigcup_{C\in K_J: \, y_1\cdots y_J(C)=y_1\cdots y_J} C\Sigma (y_1\cdots y_J)\vert_{L_{J+1}}.
\]
As above, we will label a cylinder  $C_{1+L_2+\cdots +L_{J+1}}$ in $C\Sigma (y_1\cdots y_J)\vert_{L_{J+1}}$  (where the cylinder $C$ is labelled   $y_1\cdots y_J(C)= y_1\cdots y_J$) by the word $y_1\cdots y_{J+1}(C_{1+L_2+\cdots +L_{J+1}})=y_1\cdots y_J0$ if and only if the cylinder  $C_{1+L_2+\cdots +L_{J +1}}$ has non-empty intersection with the interval
\begin{eqnarray*}   
\Big[&&\!\!\!\!\!\!\! \min\{ x\in C\Sigma (y_1\cdots y_J)\vert_{L_{J+1}}\},\\
&&\!    \!\!\!\! \ \!\!\!\!\!\! \ \min\{ x\in C\Sigma (y_1\cdots y_J)\vert_{L_{J+1}}\}+\frac12\diam C\Sigma (y_1\cdots y_J)\vert_{L_{J+1}} \Big],
\end{eqnarray*}
and by $y_1\cdots y_{J+1}(C_{1+L_2+\cdots +L_{J+1}})=y_1\cdots y_J 1$ else. This way we ensure that
\begin{equation}\label{diam}
\diam\left(\bigcup_{y_1\cdots y_{J+1}(C_{1+L_2+\cdots +L_{J+1}})=y_1\cdots y_{J+1}} C_{1+L_2+\cdots +L_{J+1}}\right)\le\frac1{2^{J+1}}.
\end{equation}

\begin{figure} 
  \includegraphics[width=11cm,height =9.1cm]{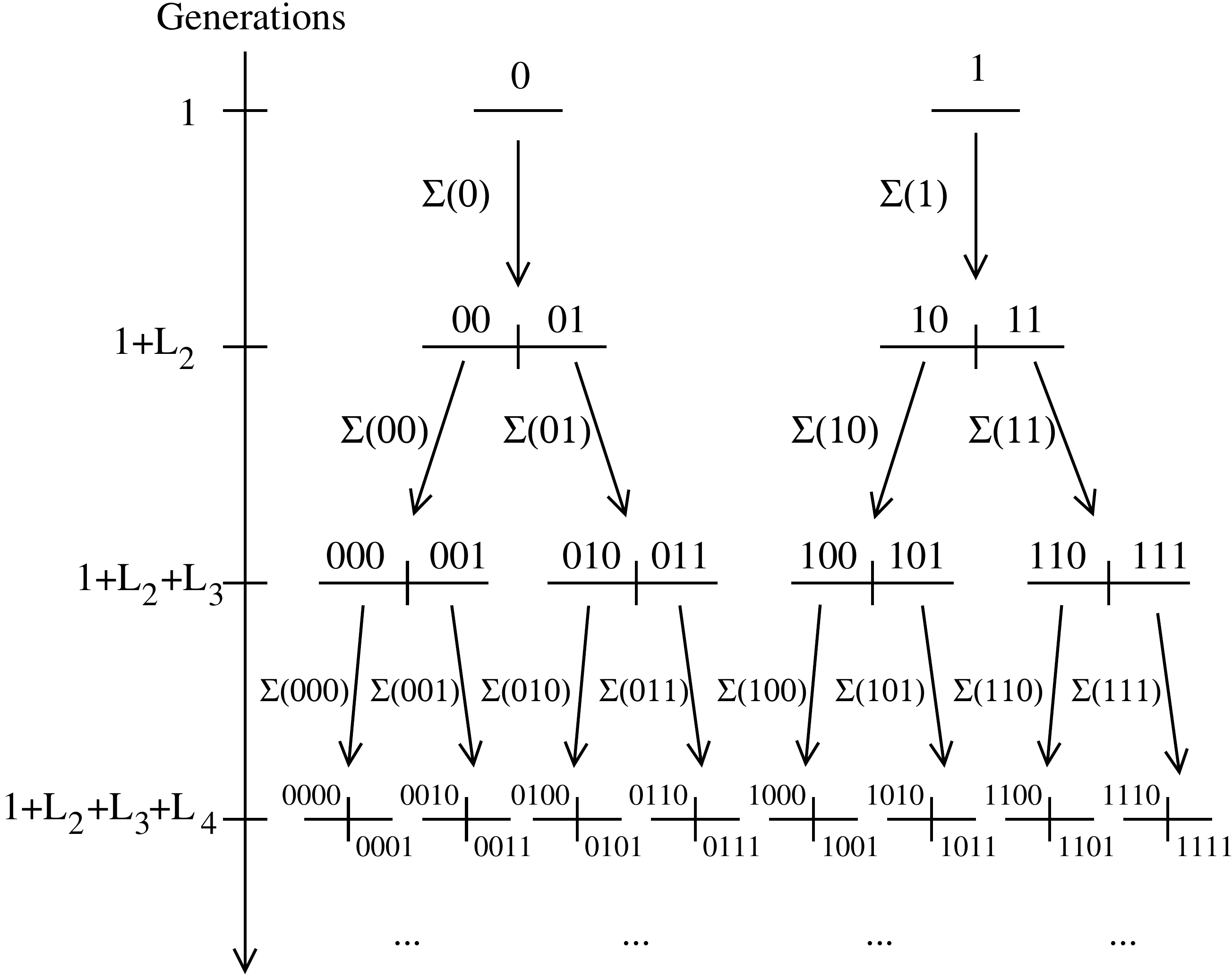} 
  \caption{Construction's scheme of the measure $\mu$.}
\end{figure}
\bigskip

{\bf Step 6:} We define the Cantor set 
\[
K:=\bigcap_{J\geq 2} K_J.
\] 
It has the following properties:
\begin{itemize}
\sk
\item[a)] $K$ is compact,
\sk
\item[b)] for $x\in K$, we have a   labelling sequence $\uy (x) =y_1\cdots y_J\cdots \in\{0,1\}^\infty$, and we will use the obvious notation $y_1\cdots y_J(x)$,
\sk
\item[c)] by the choice of the labelling and the function $\Sigma$ we have for any $x\in K$ that the limit
\[
h(x):=\lim_{J\to\infty}h_{top}(\Sigma(y_1\cdots y_J(C_{1+L_2+\cdots L_J}(x))))
\]
exists, where $C_{1+L_2+\cdots L_J}(x)$ denotes the unique dyadic cylinder of length $1+L_2+\cdots L_J$ containing $x$.

\sk
\item[d)] for $(x,x')\in K^2$, we have $\uy(x)=\uy(x')\iff x=x'$ (this is immediate from~(\ref{diam})). More precisely if $x<x'$ then $\uy(x)<\uy(x')$ (in lexicographical order) and by the choice of the function $\Sigma$
\[
h(x)<h(x').
\] 

\sk 

\item[e)] $\dim_HK=b$.
\end{itemize}
%

\sk\sk

{\bf Step 7:} We define the measure $\mu$ on the cylinder sets
\[
\left\{C\, :\, |C|=1+L_2+\cdots L_J , \ J\geq 2 \mbox{ and } C\cap K\ne\emptyset\right\}.
\]
Any such cylinder can be written as
\begin{equation}
\label{eq3}
C=C_1C_2\cdots C_J, \mbox{ where $ |C_j|=L_j$ and $ C_j\cap\Sigma(y_1\cdots y_j(C_1\cdots C_j))\ne\emptyset $}.
\end{equation}
Then we set
\[
\mu(C)\, :=  \, \frac12\prod_{j=2}^J   \mu_{\Sigma(y_1\cdots y_j(C_1\cdots C_j))}(C_j).
\]
This is clearly a ring of subsets and hence by Caratheodory's extension theorem we get a measure on $[0,1]$ with support $K$.
It has the following properties:
\begin{itemize}
\sk
\item[a)] $\supp\mu) =K$,
\sk
\item[b)] for $x\in K$ we have 
\[
d_\mu(x)=\frac{h(x)}{\log2}\in [a,b],
\]
\sk
\item[c)] for $I\cap K\ne\emptyset$ with $I$ an interval we have that $\mu(I)>0$,

\sk
\item[d)] From item d) in {\bf Step 6} combined with the previous item,  if $(x,x')\in K^2$ and  $x<x'$, then  
\[
d_\mu(x)<d_\mu(x').
\]
 Hence $D_\mu(\a)$ consists of at most one point.

\sk
\item[e)] $\dim_H\mu=b$ since $\esup d_\mu=b$.
\end{itemize}
\sk

In the above statements, only item b) needs an explanation. Once it will be proved, items c), d) and e) will follow directly using obvious arguments.

\begin{proposition}
For every  $x\in K$, $\displaystyle d_\mu(x)= \frac{h(x)}{\log 2}$.
\end{proposition}
\begin{proof}
The point is to prove that the liminf used when defining $d_\mu(x)$ is in fact a limit, and that it coincides with $h(x)$.

Let us first prove that 
\begin{equation}
\label{eq2}
 \frac{\log \mu( C_{1+L_2+\cdots L_J}(x)  ) }{ - \log_2( 1+L_2+\cdots L_J)}  \  {\longrightarrow} \ h(x)
 \end{equation}
 when  $J\to +\infty$.
Once (\ref{eq2}) will be proved, we will have to take care of the generations between $1+L_2+\cdots L_J$ and $1+L_2+\cdots L_{J+1}$.

Let $J\geq 1$.  We use the decomposition  (\ref{eq3}) of  the cylinder $C_{1+L_2+\cdots L_J}(x)$. By our choice for $L_J$ in Step 4, we have
\begin{eqnarray*}
\mu( C_{1+L_2+\cdots L_J}(x)  ) & = & \frac{1}{2}\, \prod_{j=2}^J   \mu_{\Sigma(y_1\cdots y_j(x))}(C_j)\\
&\leq &  \prod_{j=2}^J 2^{ - \big(  h_{top}\big(\Sigma(y_1\cdots y_j(x))\big)- \delta_ { j} \big ) L_i}\\
&\leq &  {2^{ - \big(   h_{top}\big(\Sigma(y_1\cdots y_J(x))\big)- \delta_ { J} \big) L_J} }   2^{- P_JL_J  },
\end{eqnarray*}
where 
\begin{eqnarray*}
P_J := \sum_{j=2}^{J-1}  \big(   h_{top}\big(\Sigma(y_1\cdots y_J(x)) \big)-\delta_ { j} \big) \frac{L_j}{L_J}  \, \geq \,   \sum_{j=2}^{J-1}   a \frac{L_j}{L_J}  \,  \geq  \,\delta_ {J},
\end{eqnarray*}
the last inequality following from Step 4 and the definition of $\delta_J$. Hence,
\begin{eqnarray}
\label{eq4}
\mu( C_{1+L_2+\cdots L_J}(x)  ) &\leq &  {2^{ - \big(   h_{top}\big(\Sigma(y_1\cdots y_J(x))\big)-  2\delta_ {J}\big) L_J} }   .
\end{eqnarray}

The same inequality in Step 4 ensures that $|C_{1+L_2+\cdots L_n}(x)  | = 2^{-(1+L_2+\cdots L_J)} $ is upper and lower-bounded respectively by $2^{-L_J (1 - \delta_ J)}$ and  $2^{-L_J (1 + \delta_ J)}$. We deduce that
\begin{equation}
\label{majmin1}
\mu( C_{1+L_2+\cdots L_J}(x)  )  \leq |C_{1+L_2+\cdots L_J}(x)  | ^{ \big(   h_{top}\big(\Sigma(y_1\cdots y_J(x))\big)- 2\delta_ {J} \big)\big(1 -\delta_ {J}\big) }.
\end{equation}
The same arguments yield the converse inequality
\begin{equation}
\label{majmin2}
\mu( C_{1+L_2+\cdots L_J}(x)  )  \geq |C_{1+L_2+\cdots L_J}(x)  | ^{   \big(   h_{top}\big(\Sigma(y_1\cdots y_J(x)\big) +2\delta_{J} \big)\big(1 +\delta_ {J}\big) },
\end{equation}
and taking logarithms, (\ref{eq2}) follows.

\mk

Let now $n$ be an integer in $\{1, \cdots,  L_{J+1}-1\}$, and consider $C_{1+L_2+\cdots L_J +n}(x)$. We write $C_{1+L_2+\cdots L_J +n}(x) = C_1\cdots C_J C_{J+1}$ with $|C_j|=L_j$ for every $j \leq J$, and $|C_{J+1}|=n$.

\begin{itemize}

\sk\item
{\bf If $1 \leq n \leq N_{J+1}$:}  we get 
\begin{eqnarray*}
\mu(C_{1+L_2+\cdots L_J +n}(x))  & = &  \frac{1}{2}\prod_{j=2}^{J+1}\mu_{\Sigma(y_1\cdots y_j(x))}(C_j)\\
& = &  \frac{1}{2}\prod_{j=2}^{J-1}\mu_{\Sigma(y_1\cdots y_j(x))}(C_j)\\
&& \ \ \  \times \mu_{\Sigma(y_1\cdots y_J(x))}(C_J) \cdot \mu_{\Sigma(y_1\cdots y_{J+1} (x))}(C_{J+1})\\
& \leq &{2^{ - \big(   h_{top}\big(\Sigma(y_1\cdots y_{J-1}(x))\big)- 2\delta_ {{J-1}}\big) L_{J-1}} }  \\
&& \ \ \  \times 2^{-(h_{top}(\Sigma(y_1\cdots y_J (x)))- \delta_ { J})(L_J+n)},
\end{eqnarray*}
where (\ref{eq4}) and (\ref{eq1}) have been used to bound from above respectively the first  and the second product.

Using the same arguments as above, we see that 
\begin{eqnarray}
\label{majmin3,5}
\mu(C_{1+L_2+\cdots L_J + n}(x))  
&  \leq  & |C_{1+L_2+\cdots L_J + n }(x))| ^{h_{top}(\Sigma(y_1\cdots y_J(x)))- \delta'_J} ,
\end{eqnarray}
where $(\delta'_J)_{J\geq 2}$ is some other positive sequence converging to zero when $J$ tends to infinity.  
\

\sk\item
{\bf If $N_{J+1}+1  \leq  n \leq L_{J+1} -1$:}  we have 
 \begin{eqnarray*}
\mu(C_{1+L_2+\cdots L_J +n}(x))  & = &  \frac{1}{2}\prod_{ j =2}^{J+1}\mu_{\Sigma(y_1\cdots y_j (x))}(C_j)\\
& = &  \frac{1}{2}\prod_{j=2}^{J}\mu_{\Sigma(y_1\cdots y_j(x))}(C_j)  \times  \mu_{\Sigma( y_1\cdots y_{J+1} (x))}(C_{J+1})\\
& \leq &{2^{ - \big(   h_{top}\big(\Sigma( y_1\cdots y_J (x))\big) - 2\delta_ {{J }}\big) L_J} } \\
&&  \cdot  2^{-(h_{top}(\Sigma(y_1\cdots y_{J+1} (x)))-\delta_  {J+1}) n },
\end{eqnarray*}
where (\ref{eq4}) and Step 2 of the construction have been used to bound from above respectively the first  and the second product.

Using the same arguments as above, we see that 
\begin{equation}
\label{majmin4}
\mu(C_{1+L_2+\cdots L_J + n }(x))  
  \leq   |C_{1+L_2+\cdots L_J + n}(x))| ^{ h_{J,n}} ,
\end{equation}
where $ h_{J,n} $  is a real number  between $   h_{top}\big(\Sigma(y_1\cdots y_J(x))\big)-2\delta_ {{J }} $ and $ h_{top}(\Sigma(y_1\cdots y_{J+1} (x)))-\delta_ { {J+1}} $, which gets closer and closer  to  the exponent  $ h_{top}(\Sigma(y_1\cdots y_{J+1} (x)))- \delta_ { {J+1}}   $ when $n$ tends to $L_{J+1}$.  

In particular, $h_{J,n}$ converges to $h(x) $ when $J$ tends to infinity, uniformly in $n\in \{1, \cdots,  L_{J+1}-1\}$.

\sk\item
The converse inequalities are proved using the same ideas.

\end{itemize}
\end{proof}

To finish the proof of Theorem \ref{main}, we make the following observations.

By construction, we see that the support $S$ of the multifractal spectrum of $\mu$  is actually the image of the middle-third Cantor set by the map $\alpha \mapsto a+ (b-a)\alpha$. We deduce that $S \subset [a,b]$, $\min(S)=a$ and $\max(S)=b$, and that $D_\mu(\alpha)$ contains either 0 or 1 point, for every $\alpha \geq 0$.  This   proves  parts iii) and iv) of Theorem \ref{main}, and also part i) of Theorem \ref{main2}.

%

\section{The large deviations and the Legendre spectra}
\label{SecLD}

We prove Theorem \ref{main2}.  
 
Recall that the Cantor set $K$ is the support of $\mu$ and that $S =\{ d_\mu(x):x\in K\}$ is the image of the middle-third Cantor set by an affine map.

\subsection{The large deviations spectrum}

First, let $\alpha \in   S$, and  let $x_\alpha$ be the unique point such that $d_\mu(x_\alpha)=\alpha$. One will use the labelling $y_1\cdots y_j(x_\alpha)$, since by construction  one has $\alpha = \lim_{j\to +\infty} a+(b-a)  \times 0,y_1\cdots y_j(x_\alpha)$.

Let $\e>0$. Due to our construction, there exists a real number $\eta(\e)$, that converges to zero when $\e$ tends to zero, such that  $|h_{top}(\Sigma(y_1\cdots y_j(x))) -\alpha | \leq 2\e$ implies that $|x-x_\alpha|\leq \eta(\e)$.

By construction, there exists a generation $J_\e$ such that for every $j\geq J_\e$, $|h_{top}(\Sigma(y_1\cdots y_j(x_\alpha) )) - \alpha|\leq \e  $. Moreover,  $J_\e$ can be chosen large enough that $\delta_{J_\e} \leq \e /2$.

Observe that if $\tilde C$ is a cylinder of generation $j\geq J_\ep$ such that  
\begin{equation}
\label{ineg11}
 |\tilde C| ^{\alpha+\e}   \leq  \mu(\tilde C) \leq  |\tilde C| ^{\alpha-\e} ,
\end{equation}  
is satisfied, then 
by (\ref{majmin3,5}), (\ref{majmin4}) and our choice for $J_\ep$,  $\tilde C$ is necessarily included in a cylinder $ C$ of generation $J_\ep$ such that 
\begin{equation}
\label{ineg10}
|y_1\cdots y_{J_\ep}(x_\alpha) 
- y_1\cdots y_{J_\ep}(C)| \leq \eta(\ep).
\end{equation}

 Hence, to bound by above the number $N_j(\alpha,\ep)$ (defined by (\ref{defNJ})), it is sufficient to count the number of cylinders $\tilde C$ of generation $j$  included in    the  cylinders $ C$ of generation $J_\ep$  such that (\ref{ineg10}) holds. 
 
 Let us denote by $M_{\alpha,\ep} $ the number of   cylinders $ C$  of generation ${J_\ep}$  satisfying (\ref{ineg10}), and fix $C_{J_\ep}$  such a cylinder.
 
 Obviously, all the subshifts of finite type $\Sigma$ which are used in the construction of $K$ inside $C_{J_\ep}$ have a topological entropy   which satisfies $|h_{top}(\Sigma) -\alpha| \leq 2\e$. Hence, it is an easy deduction of the preceding considerations that  the number of cylinders of generation $j$ included in $C_{J_\ep}$ is  lower- and upper-bounded by
 $$ 2^{ (\alpha-2\ep) j}  <  \# \{C\in \mathcal{G}_j: C\subset C_{J_\ep}  \mbox{ and }  C \cap K \neq \emptyset \}< 2^{(\alpha+2\e)j} . $$
 Consequently, 
 $$  N_j(\alpha,\ep) \leq M_{\alpha,\ep}  2^{(\alpha+2\e)j}.$$
 Taking the liminf of $\displaystyle \frac{\log_2 N_j(\alpha,\ep) }{j}$ when $j$ tends to infinity, and letting $\ep$ go to zero, we find that $LD_\mu(\alpha) \leq  \alpha$. 
  
 \mk
 
 One gets the lower bound using what precedes. Indeed,  in the above proof, all the cylinders  $C\in \mathcal{G}_j$ satisfying $C\subset C_{J_\ep}  \mbox{ and }  C \cap K \neq \emptyset \}$  verify $$ |C|^{\alpha+3\ep} \leq \mu(C) \leq  |C|^{\alpha-3\ep}. $$
  Hence 
 $$ M_{\alpha,\ep}   2^{ (\alpha-2\ep) j}  \leq N_j(\alpha,3\ep).$$
 By taking a liminf and letting $\ep$ go to zero, we get  that $LD_\mu(\alpha) \geq  \alpha$.

\mk

If $\alpha \notin S$, then there exists $\ep>0$ such that $[\alpha-2\ep,\alpha+2\ep]\cap S =\emptyset$. Hence, using again  (\ref{majmin3,5}), (\ref{majmin4}) and  choosing $J$ sufficiently large so that $\delta_J\leq \ep/2$, one sees that for every cylinder $C$ of generation $j\geq J_\ep$ such that $C\cap K \neq \emptyset$, $\mu(C)\notin [ |C|^{\alpha+\ep} , |C|^{\alpha-\ep}]$. Consequently, $N_j(\alpha,\ep)=0$ and $LD_\mu(\alpha)=-\infty$.

\subsection{The Legendre  spectrum}

\

Finally, we compute the Legendre spectrum. Obviously $\tau_\mu(1)=0$, and $\tau_\mu(0)= \dim_B
 \mu =b$, where $\dim_B$ stands for the Minkovski dimension.This is actually relatively easy with what precedes. Indeed, we  proved that for every $\ep>0$,
if $j$ is large enough, then all cylinders $C$ of generation $j$ such that $C\cap K \neq \emptyset$ satisfy
$$2^{- j (b+\ep)} \leq \mu(C) \leq 2^{-j(a-\ep)}.$$

Let us cover the set $S=\{\alpha\geq 0: D_\mu(\alpha)\neq \emptyset\}$ by a finite set of intervals $(I_n)_{n=1,\cdots, N}$ of the form $I_n= [\alpha_n-\ep,\alpha_n+\ep]$, where for every $n\in \{1,2,\cdots, N\}$, $\alpha_n \in S$, and $\alpha_1=a$ and $\alpha_N=b$.   For every $n$, the estimates above yield that if $j$ is large,
$$ 2^{j(\alpha_n-\ep_n)} \leq N_j(\alpha_n,\ep) \leq 2^{j(\alpha_n+\ep_n)},$$
where $\ep_n$ is some positive real number converging to zero when $\ep$ goes to zero. Hence we find that for $q>0$,
$$ \sum_{n=1}^N 2^{j (\alpha_n - \ep_n)}  2 ^{-qj(\alpha_n+ \ep)}     \leq  \sum_{C\in \mathcal{G}_j}  \ \mu(C)^q  \leq \sum_{n=1}^N 2^{j(\alpha_n+\ep_n)}2^{-qj(\alpha_n-\ep)}.$$
If $q>1$,  then  the right hand-side term is equivalent to $2^{j(a(1-q) +\ep_1 +q\ep)}$, and the left hand-side term is equivalent to $2^{j(a(1-q) -\ep_1 -q\ep)}$. Hence, by taking liminf when $j$ tends to infinity, we obtain $\tau_\mu(q) = a(q-1)$.

If $q\in (0,1)$, then  the right hand-side term is equivalent to $2^{j(b(1-q) +\ep_N +q\ep)}$, and the left hand-side term is equivalent to $2^{j(b(1-q) -\ep_N-q\ep)}$. We deduce that  $\tau_\mu(q) = b(q-1)$.

Finally, when  $q<0$ one has 
$$ \sum_{n=1}^N 2^{j (\alpha_n - \ep_n)}  2 ^{-qj(\alpha_n+ \ep_n)}     \leq  \sum_{C\in \mathcal{G}_j}  \ \mu(C)^q  \leq \sum_{n=1}^N 2^{j(\alpha_n+\ep_n)}2^{-jq(\alpha_n-\ep_n)}.$$
The same estimates yield that $\tau_\mu(q)= b(q-1)$.

\end{document}